\documentclass[12pt]{article}

\usepackage[a4paper, margin=2cm]{geometry}
\usepackage[pagewise]{lineno}
\usepackage[pdftex]{graphicx}
\usepackage{graphicx}
\usepackage{amssymb}
\usepackage{amsmath,amsthm}
\usepackage{epstopdf}
\usepackage{subfigure}
\usepackage{tikz}
\usepackage{tkz-graph}
\usepackage{mathtools}
\usepackage{csquotes}
\usepackage[bb=boondox]{mathalfa} 
\usepackage{lineno}

\newtheorem{theorem}{Theorem}[section]
\newtheorem{corollary}[theorem]{Corollary}
\newtheorem{lemma}[theorem]{Lemma}
\newtheorem{proposition}[theorem]{Proposition}

\newtheorem{definition}{Definition}
\newtheorem{example}{Example}

\newtheorem{remark}[theorem]{Remark}

\def\qed{\vbox{\hrule
  \hbox{\vrule\hbox to 5pt{\vbox to 8pt{\vfil}\hfil}\vrule}\hrule}}

\def\endproof{\unskip \nobreak \hskip0pt plus 1fill \qquad \qed \par \vspace{0.15cm}}

\begin{document}
\title{On the Eigenvectors of Generalized Circulant Matrices}

\author{
 Enide Andrade\footnote{Center for Research and Development in Mathematics and Applications,
 Department of Mathematics,  University of Aveiro, 3810-193, Portugal, {\tt enide@ua.pt}} \,
 Dante Carrasco-Olivera\footnote{Grupo de Investigaci\'on en Sistemas Din\'amicos y Aplicaciones (GISDA),
Departamento de Matem\'atica, Universidad del B\'{\i}o-B\'{\i}o, Avda. Collao 1202, Concepci\'on, Chile, {\tt dcarrasc@ubiobio.cl}} \,
 Cristina Manzaneda\footnote{Departamento de Matem\'{a}ticas, Facultad de Ciencias. Universidad Cat\'{o}lica del Norte. Av. Angamos 0610 Antofagasta, Chile, {\tt cmanzaneda@ucn.cl} Corresponding author.}
 }

 \date{May 15, 2023}

\maketitle

\begin{abstract}
In this paper, closed formulas for the eigenvectors of a particular class of matrices generated by generalized permutation matrices, named generalized circulant matrices, are presented. 

\end{abstract}



\section{Introduction}

In \cite{Mourad}, Kaddoura and Mourad, in order to widen the scope of the class of circulant matrices, (see \cite{Davis}), constructed circulant-like matrices that were called generalized weighted circulant matrices. 
These matrices form a class of matrices generated by generalized permutation matrices corresponding to a subgroup of some permutation group. The characteristic polynomials, eigenvalues and eigenvectors of the generalized permutation matrices corresponding to a family of permutations were described. Additionally, the eigenvalues of the weighted circulant matrices were given however, its eigenvectors were not studied. Having these results as motivation, we present, in some cases, explicit formulas for the eigenvectors of the generalized weighted circulant matrices. In this work, they are simply called generalized circulant matrices.

\medskip
\noindent {\em Notation}:  
$\mathbb{C}$ is the field of complex numbers 
and the imaginary unit is denoted by \rm{i}. Moreover, $\mathbb{N}$ represents the set of natural  numbers. The identity matrix of order $m$ is denoted by $I_{m},$ and $\mbox{\rm{diag}}( a_{11}, \ldots, a_{mm})$ represents the diagonal matrix with diagonal entries $a_{11}, a_{22}, \ldots, a_{mm}$. For any square matrix $M$, $\sigma(M)$ is its spectrum and $M^{-1}$ is its inverse. We denote by $\textbf{e}_{i}$ the $i$-th column of the identity matrix. If $M$ is any matrix, $M^{T}$ is its transpose. The symbol $\bigoplus$ represents the direct sum of matrices and, for $u=[u_1,\dots,u_m]^{T}$ and $v=[v_1,\dots,v_m]^{T}$ the Hadamard product of $u$ and $v$ is denoted by $u\odot v=[u_1v_1,\dots,u_mv_m]^T$. Moreover, $\mbox{prod}\, u=\prod_{i=1}^{m} u_i,$ and for $j\in \mathbb{N}$, $\mbox{\rm{prod}}_{j}(i)= \prod_{\ell=0}^{j-1} u_{\pi^{\ell}_{s}(i)}.$ 
Additionally, $F=(\omega^{i(j-1)})$, $1\leq i,j\leq m,$
is the discrete Fourier transform, where $i(j-1) \equiv r (\mbox{\rm{mod}}\, m),$ with $r= 0, 1, \ldots, m-1$ and $\omega=\mbox{\rm{exp}}(\frac{2\pi  \rm{i}}{m}).$ 
Also, for $a,b \in \mathbb{N}$, \rm{gcd}$(a,b)$ denotes the greatest common divisor between $a,b.$ The symmetric group of order $m$ is denoted by $S_{m},$ and the order of a permutation $\pi \in S_{m}$ is $O(\pi_{s}).$ Additionally, $\pi^{k} = \pi \circ \cdots \circ \pi.$\\ 
\\
\noindent We now present some definitions and results from \cite{Mourad} that will be used in the sequel.
Let $m \in \mathbb{N}$ and $\pi \in S_{m}$.
Each element $\pi \in S_{m}$ corresponds to a permutation matrix $P_{\pi}= (p_{i,j}),$ where $p_{i,j}=1$ if $j= \pi(i)$ and zero otherwise.
A square matrix having in each row and column only one non-zero element is called a \textit{generalized permutation matrix.}

\noindent It was stated in \cite{Mourad} that an $m\times m$ matrix $P(u, \pi)$ is a generalized permutation matrix if and only if
\begin{equation} \label{matrixU}
P(u, \pi)= D_{u} P_{\pi}, \pi \in S_{m},  
\end{equation}
where $u=\left[u_{1}, \ldots, u_{m} \right]^{T} \in \mathbb{C}^{m}$ and 
$D_{u}= \mbox{\rm{diag}}(u_{1}, \ldots, u_{m}).$ 

\noindent For $m \in \mathbb{N}$, let
$ R_{m} = \{0, 1, \ldots, m-1 \} \mbox{\, and \, } R_{m}^{\star}= R_{m}\backslash \{0 \}.$
Moreover, denote by $\mathcal{P}(R_{m})$ the group of permutations of $R_{m}$ and define
$$\Gamma_{m} = \{\pi_{s} \in \mathcal{P}(R_{m}), s \in R_{m}\},$$
where $\pi_{s}: R_{m} \rightarrow R_{m}$ and
\begin{equation} \label{pi-s}
 \pi_{s} (x) = (x+s) (\mbox{\rm{mod}}\, m).
\end{equation}
Throughout this paper we assume that $\pi_s$ is defined as in (\ref{pi-s}) and $\pi_0$ is the identity of the group $\Gamma_{m}.$ Moreover, if $k\in \mathbb{N},$ $\pi_{s}^{k}= \pi_{ks}.$

The following remark shows that for $k\in \mathbb{N}$, the matrix $P(u,\pi_s)^{k}$ is also a generalized permutation matrix.\\
 
\begin{remark}  {\rm \cite{Mourad} \label{remark}  Let $P(u, \pi_{s})$ be a generalized permutation matrix where $s \in R_{m}^{\star}$. If $O(\pi_{s})=m$, from
\cite[Corollary 1.12]{Mourad} we have $P(u, \pi_{s})^{m}= (\mbox{prod}\,u) I_{m}$. 
Additionally, from \cite[Corollary 1.3]{Mourad}, if 
$k \in \mathbb{N}, P(u, \pi_{s})^{k} = P(v_{k}, \pi_{s}^{k}),$ and
\begin{equation}\label{v}
v_k = u\odot \bigodot_{i=1}^{k-1} \pi_{s}^{i}(u).
\end{equation}

\noindent Then, from \cite[Corollary 1.13]{Mourad}, if $k \geq m$ and $k= qm+r,$ for some $q\in \mathbb{N}$, such that $0\leq r< m$, 
$$P(u, \pi_{s})^{k}= (\mbox{prod}\,u)^{q}P(v_r, \pi^r_{s}).$$ }
\end{remark}

\begin{definition} {\rm \cite{Mourad} }\label{Mouradref} 
A generalized circulant matrix corresponding to $P(u, \pi_{s})$ is
\begin{equation} \label{matrixC} C(u, \pi_s)= \sum_{r=0}^{k} c_{r} P(u, \pi_{s})^{r},\end{equation} where $k \in \mathbb{N}$ and $c_{r} \in \mathbb{C}$.
\end{definition}
\noindent  The following theorem gives an explicit expression for the eigenvalues of $C(u, \pi_s)$. The formula (\ref{eigenvalues1}) has a misprint in the original paper, \cite[Theorem 1.17]{Mourad} so, we correct it here. 

\begin{theorem}\label{TheoremMourad}{\rm \cite{Mourad}} Let $C= C(u, \pi_s)$ be as in (\ref{matrixC}), where $s \in R_{m+1}^{\star}$, $d= O(\pi_{s}),$ and $g=$\rm{gcd}$(m,s).$ Then, the eigenvalues of $C$ are given by:
\begin{equation}\label{eigenvalues1}
\lambda_{t,p}(C) = \sum_{r=0}^{k} c_{r} \left( \mbox{\rm{prod}}_d (t) \right)^{\frac{r}{d}}\mbox{ \rm{exp}}\left( \frac{2\pi p r\rm{i}}{d} \right)
\end{equation}
where $p= 0, 1, \ldots, d-1$ and $t=1,2, \ldots, g$. In particular if $O(\pi_{s})=m$, the eigenvalues of $C$ are simply given by
\begin{equation*}
\lambda_{p}(C)= \sum_{r=0}^{k} c_{r} (\mbox{prod}\,u))^{\frac{r}{m}} \mbox{\rm{exp}} \left( \frac{2\pi  p r\rm{i}}{m} \right)
\end{equation*}
with $p=0,1, \ldots, m-1.$ 
\end{theorem}

\noindent Additionally, from \cite[Theorem 1.14]{Mourad} for a generalized permutation matrix $P(u, \pi_{s})$, $d=O( \pi_{s})$ and $g=$\rm{gcd}$(m,s),$ the eigenvalues of $P(u, \pi_{s})^{d}$ are given by $$ \lambda_{i} (P(u, \pi_{s})^{d})= \mbox{\rm{prod}}_d (i) , \mbox{\,for  \, \,} i= 1, 2,\ldots, g,$$
where each $\lambda_i$ is repeated $d$ times, and its corresponding eigenvectors $V_{i}^{(t)}$, are the following:
\begin{equation}\label{eigv} V_{i}^{(t)}(P(u, \pi_{s})^{d}) = (\mbox{\rm{prod}}_d (i)) \textbf{e}_{i+tg},
\end{equation}
for each $t=1, \ldots,d$.\\

\noindent The aim of this paper is to present explicit formulas for the eigenvectors of $C(u, \pi_s)$, for the cases:

\begin{itemize}
    \item $\pi_{s} \in \Gamma_{m},$ when $s=1;$ for any $k \in \mathbb{N}$.
    \item $\pi_{s} \in \Gamma_{m},$ when $s\geq 2$ and \rm{gcd}$(m,s)=1;$ for any $k\in \mathbb{N}$. 
    \item $\pi_{s} \in \Gamma_{m}, \mbox{ when\,} s | m $. 
    
    \end{itemize}

\section{Eigenvectors of $C(u, \pi_s)$}

In this section, explicit formulas for the eigenvectors of $C(u, \pi_s)$, are given for the cases presented in the end of previous section. 
Throughout this text we assume that all $u_i$, $ i=1, \ldots,m$, are nonzero. Note that, if $u_i=0$ for some $i=1,2,\ldots, m$, then the matrix $P(u, \pi_{s})$ would not be a generalized permutation matrix. 

\subsection{\bf Case $s=1.$}

The eigenvectors of  $C(u,\pi_1)$ are presented next. 
\begin{proposition} \label{Prop0}
\noindent Let $u_1, \ldots, u_n \neq 0$, $C(u, \pi_1)$ be the matrix of order $m$ as in (\ref{matrixC}) corresponding to $P(u, \pi_1)$. Let 
$\lambda=\sqrt[m]{\mbox{\rm{prod}}\,u},$  $\Lambda_1=\mbox{\, \rm{diag}\,}\left(\dfrac{\lambda}{u_m}, \dfrac{\lambda^2}{u_1u_m},\ldots,\dfrac{\lambda^{m-1}}{u_1u_2\cdots u_{m-2}u_m},1\right)$ and $F$ be the discrete Fourier transform. Then, the columns of the matrix $\Lambda_1F$ form a basis of eigenvectors of $C(u, \pi_1)$.
\end{proposition}

\begin{proof}
 By Theorem \ref{TheoremMourad},
\begin{equation} \label{e}\sigma(C(u, \pi_1))=
\left\lbrace\sum_{r=0}^{k}c_r\lambda^r,\sum_{r=0}^{k}c_r(\lambda\omega)^r,\ldots,\sum_{r=0}^{k}c_r(\lambda\omega^{(m-1)})^r\right\rbrace.
\end{equation}
\noindent Note that the matrix $P(u, \pi_{1})$ is diagonalizable since its eigenvalues are distinct (\cite[Corollary 1.15]{Mourad}).
Let
\[T=\Lambda_1 F=\left(\omega^{i(j-1)}\frac{\lambda^{j}}{u_1u_2\cdots u_{j-1}u_m}\right)\,\,,1\leq i,j\leq m.\]
\noindent For $j=1,2,\ldots,m$, let $t(j)$ be the $j$-th column of the matrix $T$ as follows:
$$
t(j)=\left(\begin{array}{c}
    \omega^{(j-1)}\frac{\lambda}{u_m}  \vspace{0.3cm}\\
    \omega^{2(j-1)}\frac{\lambda^2}{u_1u_m}\\
    \vdots\\
    \omega^{(m-1)(j-1)}\frac{\lambda^{(m-1)}}{u_1u_2\cdots u_{m-2}u_m}\\
    1\\
\end{array}\right).
$$
Thus, 
$$
P(u, \pi_{1})t(j)=\lambda\omega^{(j-1)}t(j).$$
Consequently, $t(j),$ is an  eigenvector of $P(u, \pi_{1})^{r}$ associated to $\lambda^{r}\omega^{(j-1)r}$. Then the claim follows easily.   
\end{proof}

\noindent In the next example, using Theorem \ref{TheoremMourad}, we present the eigenvalues and eigenvectors of $P(u, \pi_{1}),$ with $u=[u_1,u_2,u_{3}]^{t}.$ Additionally, from Proposition \ref{Prop0} the eigenvectors of a particular $C(u, \pi_1)$ are given. 

\begin{example} {\rm
Let $u=[u_1,u_2,u_3]^{T}$ with $u_1, u_2, u_3 \neq 0$, and $P(u, \pi_{1})=\left(\begin{array}{ccc}
           0  & u_1 &  0 \vspace{0.5cm}\\
           0  & 0   & u_2 \vspace{0.5cm}\\
           u_3& 0   &  0 \\
\end{array}\right)$, be a matrix as in (\ref{matrixU}). Using Theorem \ref{TheoremMourad}, its spectrum is $$\{\sqrt[3]{u_1u_2u_3},-\frac{1}{2}(1+\rm{i}\sqrt{3})\sqrt[3]{u_1u_2u_3},-\frac{1}{2}(1-\rm{i}\sqrt{3})\sqrt[3]{u_1u_2u_3} \},$$
and the corresponding eigenvectors are:
\begin{eqnarray*}
V^{(1)}_1=\left(\begin{array}{c}
    \frac{\sqrt[3]{u_1u_2u_3}}{u_3}\\
    \frac{(\sqrt[3]{u_1u_2u_3})^2}{u_1u_3}\\
    1
\end{array}\right),
V^{(1)}_2= \left(\begin{array}{c}
    -\frac{1}{2}(1+\rm{i}\sqrt{3})\frac{ \sqrt[3]{u_1u_2u_3}}{u_3}\\
    -\frac{1}{2}(1-\rm{i}\sqrt{3})\frac{(\sqrt[3]{u_1u_2u_3})^2}{u_1u_3}\\
    1
\end{array}\right),
V^{(1)}_3= \left(\begin{array}{c}
    -\frac{1}{2}(1-\rm{i}\sqrt{3})\frac{\sqrt[3]{u_1u_2u_3}}{u_3}\\
    -\frac{1}{2}(1+\rm{i}\sqrt{3})\frac{(\sqrt[3]{u_1u_2u_3})^2}{u_1u_3}\\
    1
\end{array}\right),
\end{eqnarray*}
respectively.

\noindent Now, let us consider $$C(u, \pi_1)=\mbox{\rm{i}}\, P(u, \pi_{1})^0-P(u, \pi_{1})^1+3P(u, \pi_{1})^2-\frac{1}{6}\mbox{\rm{i}}\, P(u, \pi_{1})^3+\frac{1}{2}P(u, \pi_{1})^4-\frac{1}{2}P(u, \pi_{1})^5,$$ when $u_1=-2;\,\,u_2=-3;\,\, u_3=1$ and $k=5$. That is,
$$
C(u, \pi_1)=\left(\begin{array}{ccc}
           0  & -4 &  0  \vspace{0.5cm}\\
           0  & 0   & -6 \vspace{0.5cm}\\
           2  & 0   &  0 \\
\end{array}\right).
$$
Then, from Theorem \ref{TheoremMourad}, $\sigma(C(u, \pi_1))=\{2\sqrt[3]{6},-(1+\rm{i}\sqrt{3})\sqrt[3]{6},-(1-\rm{i}\sqrt{3})\sqrt[3]{6} \}.$ From Proposition \ref{Prop0}, the columns $t(1),t(2), t(3)$ below form a basis of eigenvectors for $C(u, \pi_1):$
$$
t(1)=\left(\begin{array}{c}
    \sqrt[3]{6}\vspace{0.4cm}\\
    -\frac{\sqrt[3]{36}}{2}\vspace{0.4cm}\\
    1\\
\end{array}\right);\,\,\,
t(2)=\left(\begin{array}{c}
    \frac{-(1+\rm{i}\sqrt{3})\sqrt[3]{6}}{2}\vspace{0.4cm}\\
    \frac{(1-\rm{i}\sqrt{3})\sqrt[3]{36}}{4}\vspace{0.4cm}\\
    1\\
\end{array}\right)
;\,\,\,
t(3)=\left(\begin{array}{c}
    \frac{(1-\rm{i}\sqrt{3})\sqrt[3]{6}}{2}\vspace{0.4cm}\\
    \frac{(1+\rm{i}\sqrt{3})\sqrt[3]{36}}{4}\vspace{0.4cm}\\
    1\\
\end{array}\right).
$$
Note that, the spectrum of $C(u, \pi_1)$ coincide with the one determined in Theorem  \ref{TheoremMourad} however, the explicit expression for the eigenvectors is given here. 
\endproof
}

\end{example}

\begin{proposition}\label{Propremark}
Let $C(u, \pi_1)$ be as in (\ref{matrixC}) for $s=1$. Then, $C(u, \pi_1)$
can be expressed as a linear combination of the matrices 
$$I_m, P(u, \pi_{1}),P(u, \pi_{1})^{2},\ldots,P(u, \pi_{1})^{m-1}.$$
\end{proposition}

\begin{proof} We split the proof into two cases. \\
\textbf{Case $k< m-1$:} Then, $C(u, \pi_1)=\sum_{r=0}^{m-1}c_{r}P(u, \pi_{1})^r$, where $c_{k+1}=c_{k+2}=\cdots=c_{m-1}=0$. \\
\textbf{Case $k\geq m$:} In this case we start to consider $k=qm-1, $ with $q \geq 1.$ Thus,
\begin{eqnarray*}
C(u, \pi_1)& =&  \sum_{r=0}^{m-1}c_{r}
P(u, \pi_{1})^{r}+ \sum_{r=m}^{2m-1}c_{r}P(u, \pi_{1})^{r}+ \sum_{r=2m}^{3m-1}c_{r}P(u, \pi_{1})^{r}
+ ...\\
&+& \sum_{r=(q-1)m}^{qm-1}c_{r}P(u, \pi_{1})^{r}.
\end{eqnarray*}
Therefore,
\begin{eqnarray*}
C(u, \pi_1)& =& \sum_{r=0}^{m-1}c_{r}P(u, \pi_{1})^{r}+ \sum_{r=0}^{m-1}c_{m+r}P(u, \pi_{1})^{m+r}+ \sum_{r=0}^{m-1}c_{2m+r}P(u, \pi_{1})^{2m+r}\\
&+& \sum_{r=0}^{m-1}c_{3m+r}P(u, \pi_{1})^{3m+r}+ \cdots
+ \sum_{r=0}^{m-1}c_{(q-1)m+r}P(u, \pi_{1})^{(q-1)m+r} .
\end{eqnarray*}
From Remark \label{Remark}, as $P(u, \pi_{1})^{\theta m +r} = (\mbox{prod}\,u)^{\theta} P(u, \pi_{1})^{r}, \theta \in \mathbb{N} $ we have:
\begin{align*}
C(u, \pi_1)&=\sum_{r=0}^{m-1}c_{r}P(u, \pi_{1})^r+\sum_{r=0}^{m-1}(\mbox{prod}\,u) c_{m+r}P(u, \pi_{1})^{r}\\
&+\sum_{r=0}^{m-1}(\mbox{prod}\,u)^2c_{2m+r}P(u, \pi_{1})^r+\cdots
+\sum_{r=0}^{m-1}(\mbox{prod}\,u)^{q-1}c_{(q-1)m+r}P(u, \pi_{1})^r\\
&=\sum_{r=0}^{m-1}(c_r+(\mbox{prod}\,u)c_{m+r}+\cdots+(\mbox{prod}\,u)^{q-1}c_{(q-1)m+r})P(u, \pi_{1})^r\\
& =\sum_{r=0}^{m-1} (\sum_{j=0}^{q-1}c_{jm+r}(\mbox{prod}\,u)^{j}) P(u, \pi_{1})^{r}=\sum_{r=0}^{m-1}(p_{q-1}^{c_{jm+r}}(\mbox{prod}\,u))P(u, \pi_{1})^r,
\end{align*}
\medskip
where $\sum_{j=0}^{q-1}c_{jm+r}(\mbox{prod}\,u)^{j}=p_{q-1}^{c_{jm+r}}(\mbox{prod}\,u).$\\

\noindent Thus, if $k\geq m$ then $k=qm+l$, with $q \in \mathbb{N}$, and $0\leq l < m$. Note that $qm+l = (qm-1)+ l +1$. If $l+1 = m$ then $qm+l= (q+1)m-1$ and from above the result follows. 
If $l +1<  m,$ we take $c_{qm+l+2}= \cdots = c_{qm+ (m-l)}=0,$ and the result is also obtained.
\end{proof}

\begin{remark} 
{\rm
From {\rm \cite[p. 68]{Davis}} the circulant 
\begin{equation}\label{circ} \mbox{\rm{circ}} (c_0, \ldots, c_{m-1})=\sum_{r=0}^{m-1} c_{r} P_{\pi_{1}}^{r}.
\end{equation}
Then, from Proposition \ref{Propremark} the circulant can be written as a generalized circulant, as the expression in (\ref{circ}) is precisely $C(u, \pi_1)$ for $k=m-1$, $u=(u_i),$ $u_i= 1,$ for all $i= 1, \ldots, m$, and $P_{\pi_1}= P(u, \pi_1).$}
\end{remark}

\subsection{ Case $s \geq 2$ with \rm{gcd}$(s,m)=1.$}
In the next proposition we study the eigenvectors of $C(u, \pi_s)$ when  $s\geq 2$ and \rm{gcd}$(m,s)=1$. 
\begin{proposition} \label{Prop9}
Let $u_1, \ldots, u_n \neq 0$, and $C=C(u, \pi_s)$ be the matrix of order $m$ as in (\ref{matrixC}), where $s \geq 2$ and \rm{gcd}$(s,m)=1.$ Let $\lambda=\sqrt[m]{\mbox{prod}\,u}.$
Then the columns of the matrix below,

\begin{equation}\label{second}
\left(\begin{array}{ccccc}
t_1 &t_1 \omega   & t_1\omega^2   &\cdots &t_1\omega^{m-1}\\
t_2 & t_2 \omega^2 & t_2\omega^{2\cdot 2}   &\cdots &t_2\omega^{2(m-1)}\\
t_3 &  t_3 \omega^3 & t_3\omega^{3\cdot 2}   &\cdots &t_3\omega^{3(m-1)}\\
\vdots & \vdots  &   \vdots  & \ddots&  \vdots\\
t_{m-1} &  t_{m-1} \omega^{(m-1)} & t_{m-1}\omega^{(m-1)\cdot 2}   &\cdots &t_{m-1}\omega^{(m-1)(m-1)}\\
t_m & t_m & t_m & \cdots & t_m\\
\end{array}\right),
\end{equation}
form a basis of eigenvectors for $C(u, \pi_s)$ where, the first column of the matrix in (\ref{second}) is an eigenvector of $P(u,\pi_s)$ associated with $\lambda$. 
\end{proposition}

\begin{proof}
From Theorem \ref{TheoremMourad} the eigenvalues of $C$ are as in (\ref{eigenvalues1}). Consider the eigenequation 
$P(u,\pi_s)\mathbb{T}=\lambda\mathbb{T}$, where $\mathbb{T}=[t_1,t_2,t_3,\ldots,t_m]^T$. From \cite[Theorem 1.4]{Mourad}  $\mbox{\rm{det}}(P(u,\pi_s)-\lambda I)=u_1u_2\cdots u_m-\lambda^m=0$ and then the rank of the matrix $P(u,\pi_s)-\lambda I$ is less than $m$. Thus, the eigenequation have a nontrivial solution $\mathbb{T}$.
Let
$$\Lambda_s=\mbox{\, \rm{diag}\,}(t_1,t_2,\ldots,t_{m}).$$
Then
\begin{equation}
\Lambda_s F=(\omega^{i(j-1)}t_i),1\leq i,j \leq m. 
\end{equation}

Let $t(j)$ be the $j$-th column the $\Lambda_s F$ matrix, with $j=1, 
\ldots m$. Then 
\begin{align}\label{Cris}
P(u, \pi_{s})t(j)=P(u, \pi_{s})&\left(\begin{array}{c}
t_1\omega^{j-1}\\
t_2\omega^{2(j-1)}\\
\vdots\\
t_{m-s}\omega^{(m-s)(j-1)}\\
t_{m-s+1}\omega^{(m-s+1)(j-1)}\\
\vdots\\
t_m\\
\end{array}\right)=
\left(\begin{array}{c}
u_1t_{s+1}\omega^{(s+1)(j-1)}\\
u_2t_{s+2}\omega^{(s+2)(j-1)}\\
\vdots\\
u_{m-s}t_{m}\\
u_{m-s+1}t_1\omega^{j-1}\\
\vdots\\
u_m t_{s}\omega^{s(j-1)}\\
\end{array}\right)\\
&=\left(\begin{array}{c}
\lambda t_{1}\omega^{(s+1)(j-1)}\\
\lambda t_{2}\omega^{(s+2)(j-1)}\\
\vdots\\
\lambda t_{m-s}\\
\lambda t_{m-s+1}\omega^{j-1}\\
\vdots\\
\lambda t_{m}\omega^{s(j-1)}\\
\end{array}\right)
= \lambda \omega^{s(j-1)}
\left(\begin{array}{c}
t_1\omega^{(j-1)}\\
t_2\omega^{2(j-1)}\\
\vdots\\
t_{m-s}\omega^{(m-s)(j-1)}\\
t_{m-s+1}\omega^{(m-s+1)(j-1)}\\
\vdots\\
t_m\\
\end{array}\right),
\end{align}

Note that, 
$$\lambda \omega^{s(j-1)}t_{m-s}\omega^{(m-s)(j-1)}=\lambda t_{m-s}\omega^{(m-s)(j-1)+s(j-1)}=\lambda t_{m-s}\omega^{m(j-1)}=\lambda t_{m-s}$$ and $$\lambda \omega^{(j-1)s}t_{m-s+1}\omega^{(m-s+1)(j-1)}=\lambda t_{m-s+1}\omega^{m(j-1)+(j-1)}=\lambda t_{m-s+1}\omega^{j-1}.$$ Thus, the column $j$ of $\Lambda_s F$ corresponds to the eigenvector of $P(u, \pi_s)$ associated to eigenvalue $ \lambda \omega^{(j-1)s}$, for $j=1, 2,\ldots, m.$
\end{proof}

\noindent The next corollary gives closed expressions for the entries of the eigenvector $\mathbb{T}=[ t_1, \, \cdots, \, t_{m-1}, \, t_m]^{T}.$
associated to $\lambda =\sqrt[m]{\mbox{prod}\,u}$ when $s=2$ and \rm{gcd}$(m,2) =1.$\\
\begin{corollary}\label{Coro9}
Let $u_1, \ldots, u_n \neq 0$, and $C=C(u, \pi_2)$ with \rm{gcd}$(m,2)=1,$ and $ \lambda=\sqrt[m]{\mbox{\rm{prod}}\, u.}$ 
Then´, the columns of the matrix as in (\ref{second}) form a basis of eigenvectors of the matrix $C(u, \pi_2)$, where
\begin{align*}
    t_{2j+1}&=\dfrac{\lambda^{j}}{\prod_{\ell=0}^{j-1}u_{2\ell+1}}t_1,\,\, j=1,\ldots,q\\
    t_{2j}&=\dfrac{\lambda^{q+j}}{\prod_{\ell=0}^{q}u_{2\ell+1}\,\prod_{\ell=0}^{j-1}u_{2\ell}}t_1,\,\, j=1,\ldots,q\\ 
\end{align*}
with $u_0=1$.
\end{corollary}

\begin{proof}

It is clear that $m>2$ and $m$ is odd, because \rm{gcd}$(2,m)=1$. Let $m=2q+1$, $q<m$.
Let us consider  $P(u, \pi_2)$, with $u= [u_1, \ldots, u_m]^{T},$ and the eigenequation \begin{equation}\label{eigeneq}
P(u, \pi_2) \mathbb{T} =\lambda \mathbb{T}.
\end{equation}
It is easy to show that (\ref{eigeneq}) generates a system with the following pair of equations:
\begin{eqnarray}\label{eq}
    u_{1+2j}t_{2(j+1)+1}=\lambda t_{2j+1}; & \, \, 
    u_{2+2j}t_{2(j+1)+2}=\lambda t_{2j+2},
\end{eqnarray} 
for $j=0,1,2\ldots q-2$ for $q\geq 2$, 
and the three additional ones:
\begin{eqnarray*}
    u_{1+2(q-1)}t_{2q+1}&=\lambda t_{2(q-1)+1}\\
    u_{m-1}t_{1}&=\lambda t_{m-1}\\
    u_{m}t_{2}&=\lambda t_{m}.\\
\end{eqnarray*}
When $q=1$, the equations in (\ref{eq}) do not exist.
Note that the indices are calculated $\mbox{\rm{mod}}\, m$.

Solving the system based on $t_1$, we have:
\begin{align*}
    t_{2j+1}&=\dfrac{\lambda^{j}}{\prod_{\ell=0}^{j-1}u_{2\ell+1}}t_1,\,\, j=1,\ldots,q\\
    t_{2j}&=\dfrac{\lambda^{q+j}}{\prod_{\ell=0}^{q}u_{2\ell+1}\,\prod_{\ell=0}^{j-1}u_{2\ell}}t_1,\,\, j=1,\ldots,q,\\ 
\end{align*}
where $u_0=1$.

Therefore, from the Proposition \ref{Prop9}, the result follows.
\end{proof}

\begin{example}
{\rm
Consider in this example the case $s=2, m= 5$. Thus, from Proposition \ref{Prop9} and Corollary \ref{Coro9}, consider:
$$
\Lambda_2=\mbox{\, \rm{diag}\,}\left(t_1, \frac{\lambda^3 }{u_{1} u_{3} u_{5}}t_1,\frac{\lambda}{u_{1}}t_1,\frac{\lambda ^{4}}{u_{1} u_{2}u_{3}u_{5}}t_1, \frac{\lambda^2}{u_1 u_3}t_1\right).
$$
Then, taking $t_1=1,$

\begin{eqnarray*}
\Lambda_2 F &=&
\left(\begin{array}{ccccc}
1 & 0 & 0 & 0 & 0 \\
0 & \frac{\lambda ^{3}}{u_{1}u_{3}u_{5}} & 0 & 0 & 0 \\
0 & 0 & \frac{\lambda}{u_{1}} & 0 & 0 \\
0 & 0 & 0 & \frac{\lambda ^{4}}{u_{1}u_{2}u_{3}u_{5}} & 0 \\
0 & 0 & 0 & 0 & \frac{\lambda ^{2}}{u_{1}u_{3}}
\end{array}\right)\left(
\begin{array}{ccccc}
1 & w & w^{2} & w^{3} & w^{4} \\
1 & w^{2} & w^{4} & w & w^{3} \\
1 & w^{3} & w & w^{4} & w^{2} \\
1 & w^{4} & w^{3} & w^{2} & w \\
1 & 1 & 1 & 1 & 1%
\end{array}\right)\\
&=&
\left(\begin{array}{ccccc}
1 & w
& w^{2} & w^{3} & w^{4}\vspace{0.2cm}\\
\frac{\lambda ^{3}}{u_{1}u_{3}u_{5}} & w^{2}\frac{\lambda^5 }{u_1 u_{3}u_{5}} & w^{4}\frac{\lambda ^{3}}{u_{1}u_{3}u_{5}} & w\frac{\lambda ^{3}}{u_{1}u_{3}u_{5}} & w^{3}\frac{\lambda ^{3}}{u_{1}u_{3}u_{4}}\vspace{0.2cm} \\
\frac{\lambda}{u_{1}} & w^{3}\frac{\lambda}{u_{1}} & w\frac{\lambda}{u_{1}} & w^{4}
\frac{\lambda}{u_{1}} & w^{2}\frac{\lambda}{u_{1}} \vspace{0.2cm}\\
\frac{\lambda ^{4}}{u_{1}u_{2}u_{3}u_{5}} & w^{4}\frac{\lambda ^{4}}{u_{1}u_{2}u_{3}u_{5}} &
w^{3}\frac{\lambda ^{4}}{u_{1}u_{2}u_{3}u_{5}}& w^{2}\frac{\lambda ^{4}}{u_{1}u_{2}u_{3}u_{5}}
& w\frac{\lambda ^{4}}{u_{1}u_{2}u_{3}u_{5}} \vspace{0.2cm}\\
\frac{\lambda ^{2}}{u_{1}u_{3}} & \frac{\lambda ^{2}}{u_{1}u_{3}} & \frac{\lambda ^{2}}{u_{1}u_{3}} & \frac{\lambda ^{2}}{u_{1}u_{3}}& \frac{\lambda ^{2}}{u_{1}u_{3}}
\end{array}
\right)\\
&=&\left[
\begin{array}{ccccc}t(1)&  t(2)& t(3) &t(4)& t(5)
\end{array}
\right],
\end{eqnarray*}
where for $j=1,2, \ldots, 5,$ $t(j)$ denotes the column $j$ of the previous matrix.
Doing some computations, and following the formulas in (\ref{Cris}) it is easy to check that
\begin{eqnarray*}
P(u,\pi_2)T(1)& =& \lambda T(1),\\
P(u, \pi_2)T(2)&=&(\lambda w^2) T(2),\\
P(u,\pi_2)T(3)&=&(\lambda w^4) T(3),\\
P(u,\pi_2)T(4)&=&(\lambda w)T(4),\\
P(u,\pi_2)T(5)& = &(\lambda w^{3}) T(5).
\end{eqnarray*}
\endproof
}
\end{example}

\subsection{\bf Case $s|m$.}

In this subsection we study the eigenvectors of $C(u, \pi_s)$ when $s|m$. 
In this case, there exists $k_0 \in \mathbb{N}$ such that $m=k_0s$.\\

\noindent From \cite[Corollary 1.15]{Mourad} the eigenvalues of $P(u, \pi_s)$, with $g=$\rm{gcd}$(m,s),$ and $d=O( \pi_{s}),$ are given by

\begin{equation}\label{eigenvaluesU7-2}
\lambda_{t,p}=\left(\mbox{\rm{prod}}_d(t) \right)^{\frac{1}{d}}\mbox{\rm{exp}}\left(\frac{2\pi\rm{i}}{d}\right)^p
\end{equation}

where $t=1,2,\ldots,g$ and $p= 0, 1, \ldots, d-1$.\\

\noindent Note that, as $m=k_0s$, for some $k_{0}\in \mathbb{N},$ by \cite[Lemma 1.6]{Mourad}, $g=s$ and $d=k_0$. Thus, the expression in (\ref{eigenvaluesU7-2}) can be written as:

\begin{equation}\label{eigenvaluesU7-3}
\lambda_{t,p}=\left(\mbox{\rm{prod}}_{k_0} (t)\right)^{\frac{1}{k_0}}\mbox{\rm{exp}}\left(\frac{2\pi\rm{i}}{k_0}\right)^{p}=\left(\mbox{\rm{prod}}_{k_0} (t)\right)^{\frac{1}{k_0}}\omega^{p}
\end{equation}
where $t=1,2,\ldots,s$ and $p= 0, 1, \ldots, k_0-1$. \\

In this case $P(u, \pi_s)$ can be written by blocks in the following form:
\begin{equation}
\label{band_matrix2x}
 P(u, \pi_s)=\left(
\begin{array}{cccccc}
 \bf 0  & U_1  &\bf 0  & \bf 0 &\cdots & \bf 0 \\
 \bf 0  &\bf 0 &U_2    & \bf 0 &\cdots & \bf 0 \\
 \bf 0  & \bf 0&\bf 0  &U_3    &\cdots & \bf 0 \\
 \vdots &      &       &       &\ddots &\vdots \\
 \bf 0  &\bf 0 &\bf 0  &\bf 0  &\cdots & U_{k_0-1}\\
 U_{k_0}    & \bf 0&\bf 0  &\bf 0  &\cdots & \bf 0  \\
\end{array}
\right),
\end{equation}

\noindent where $$U_k={\rm diag}(u_{1+(k-1)s},u_{2+(k-1)s}, \ldots,u_{s+(k-1)s}),\, k=1,\ldots, k_{0},$$ 

\noindent and the blocks $\bf 0$ are $s$-by-$s$ matrices. 
Let
$$\mathbb{T}^{\lambda_{t,0}}=
\left[ \begin{array}{c}
\mathbb{T}^{\lambda_{t,0}}_{1}\\
\mathbb{T}^{\lambda_{t,0}}_{2}\\
\vdots \\
\mathbb{T}^{\lambda_{t,0}}_{k_0}
\end{array} 
\right], 
$$

where
\begin{align*}
\mathbb{T}^{\lambda_{t,0}}_{1}&= \textbf{e}_{t},\\
\mathbb{T}^{\lambda_{t,0}}_{2}&= \frac{\lambda_{t,0}}{\mbox{\rm{prod}}_1(t)}\textbf{e}_{t},\\
\mathbb{T}^{\lambda_{t,0}}_{3}&= \dfrac{\lambda_{t,0}^2}{\mbox{\rm{prod}}_2(t)}\textbf{e}_{t},\\
&\vdots\\
\mathbb{T}^{\lambda_{t,0}}_{k_0}&=\frac{\lambda_{t,0}^{k_0-1}}{\mbox{\rm{prod}}_{k_0-1}(t)}\textbf{e}_{t},\\
\end{align*}
where $\textbf{e}_{t}$ is the $t$-th column of the identity matrix. 
Then, 
$P(u,\pi_s)\mathbb{T}^{\lambda_{t,0}}= \lambda_{t,0} \mathbb{T}^{\lambda_{t,0}}$ as, for each $k=1,2,\ldots,k_0$, we have

\begin{align} \label{smallmatrices}
    U_k\mathbb{T}^{\lambda_{t,0}}_{k+1}&= u_{t+(k-1) s}\,\dfrac{\lambda^{k}_{t,0}}{\mbox{\rm{prod}}_k(t)}\textbf{e}_{t}= \dfrac{\lambda^{k}_{t,0}}{\mbox{\rm{prod}}_{k-1}(t)}\textbf{e}_{t}
    =\lambda_{t,0}\left( \dfrac{\lambda^{k-1}_{t,0}}{\mbox{\rm{prod}}_{k-1}(t)}\textbf{e}_{t}\right)=\lambda_{t,0} \mathbb{T}^{\lambda_{t,0}}_{k},
\end{align}
with the sub indices taken $(\mbox{\rm mod})\, k_0 $ and $\mbox{\rm prod}_0(t)=1.$
\\

\noindent Now, for each $\ell = 0,1, \ldots, k_{0}-1$, let us define the vectors:
$$\mathbb{T}^{\lambda_{t,0}}(\omega^\ell)= 
\left[ \begin{array}{c}
\omega^\ell\mathbb{T}_1^{\lambda_{t,0}}\\ 
\omega^{2\ell}\mathbb{T}_2^{\lambda_{t,0}}\\
\omega^{3\ell}\mathbb{T}_3^{\lambda_{t,0}}\\
\vdots\\
\omega^{k_0\ell}\mathbb{T}_{k_0}^{\lambda_{t,0}}
\end{array}\right].
$$

\begin{lemma}\label{lemmax} For each $t=1,\ldots,s$,
the vectors 
$$
\mathbb{T}^{\lambda_{t,0}}(\omega^{0}) , \mathbb{T}^{\lambda_{t,0}}(\omega^{1}),  \mathbb{T}^{\lambda_{t,0}}(\omega^{2}), \cdots , \mathbb{T}^{\lambda_{t,0}}(\omega^{k_0-1})
$$
are the eigenvectors of $P(u,\pi_s)$, corresponding to the eigenvalues $$\lambda_{t,0}=\lambda_{t,0}w^0, \lambda_{t,1}=\lambda_{t,0}w,\lambda_{t,2}=\lambda_{t,0}w^2,\ldots,\lambda_{t,k_0-1}=\lambda_{t,0}w^{k_0 -1},$$ respectively.
\end{lemma}

\begin{proof} As proven before, we have $P(u,\pi_s)\mathbb{T}^{\lambda_{t,0}}(\omega^{0})=\lambda_{t,0} \mathbb{T}^{\lambda_{t,0}}(\omega^{0}).$

For $\ell =1, \ldots, k_0 -1$ and, from the expressions in (\ref{smallmatrices}), assuming that $\mbox{\rm{prod}}_0(t)=1$ and $\omega^{k_0}=1,$ we have:
\begin{eqnarray*}
P(u,\pi_s)\mathbb{T}^{\lambda_{t,0}}(\omega^{\ell})&=&\left[ \begin{array}{c}
\omega^{2\ell} U_1\mathbb{T}_2^{\lambda_{t,0}}\\ 
\omega^{3\ell} U_2\mathbb{T}_3^{\lambda_{t,0}}\\
\vdots\\
\omega^{k_0\ell} U_{k_0-1}\mathbb{T}_{k_0}^{\lambda_{t,0}}\\
\omega^{\ell}U_{k_0}  \mathbb{T}_1^{\lambda_{t,0}}
\end{array}\right]
=
\left[ \begin{array}{c}
\omega^{2\ell} \lambda_{t,0}\mathbb{T}_{1}^{\lambda_{t,0}}\\ 
\omega^{3\ell} \lambda_{t,0}\mathbb{T}_{2}^{\lambda_{t,0}}\\
\vdots\\
\omega^{k_0\ell} \lambda_{t,0}\mathbb{T}_{k_0-1}^{\lambda_{t,0}}\\
\omega^{\ell}\lambda_{t,0}  \mathbb{T}_{k_0}^{\lambda_{t,0}}
\end{array}\right]\\
&=&
\lambda_{t,0}\omega^{\ell}\left[ \begin{array}{c}
\omega^\ell\mathbb{T}_1^{\lambda_{t,0}}\\ 
\omega^{2\ell}\mathbb{T}_2^{\lambda_{t,0}}\\
\omega^{3\ell}\mathbb{T}_3^{\lambda_{t,0}}\\
\vdots\\
\omega^{k_0\ell}\mathbb{T}_{k_0}^{\lambda_{t,0}}
\end{array}\right]
=(\lambda_{t,0}\omega^{\ell})\mathbb{T}^{\lambda_{t,0}}(\omega^{\ell}).
\end{eqnarray*}
Then, the result follows. 
\end{proof}

\begin{proposition}\label{eigenvectorsU}
The set 
\begin{equation*}\label{TUs}
\{\mathbb{T}^{\lambda_{t,\ell}}\,:\, t=1,2,\ldots, s\textit{ and }\ell=0,1,\ldots,k_0-1\}
\end{equation*}
forms a basis of eigenvectors of $P(u,\pi_s)$.
\end{proposition}

\begin{proof}
This result is a consequence of Lemma \ref{lemmax}.
\end{proof}

\begin{example}
{\rm
In this example, for $m=9$ and $s=3$, the eigenvectors of $P(u,\pi_3)$ corresponding to the list of eigenvalues $\lambda_{t,\ell}, t=1,2,3, \ell=0,1,2$ are presented. By the previous proposition the eigenvectors are given by:

\begin{equation*}
\mathbb{T}^{\lambda_{1,0}},\mathbb{T}^{\lambda_{1,1}},\mathbb{T}^{\lambda_{1,2}},
\mathbb{T}^{\lambda_{2,0}},\mathbb{T}^{\lambda_{2,1}},\mathbb{T}^{\lambda_{2,2}},
\mathbb{T}^{\lambda_{3,0}},\mathbb{T}^{\lambda_{3,1}},\mathbb{T}^{\lambda_{3,2}}
\end{equation*}

\begin{align*}
\mathbb{T}^{\lambda_{1,0}}&= \mathbb{T}^{\lambda_{1,0}}(\omega^0)   \\ \mathbb{T}^{\lambda_{1,1}}&=\mathbb{T}^{\lambda_{1,0}}(\omega^1)\\
\mathbb{T}^{\lambda_{1,2}}&=\mathbb{T}^{\lambda_{1,0}}(\omega^2)\\
\mathbb{T}^{\lambda_{2,0}}&=\mathbb{T}^{\lambda_{2,0}}(\omega^0)\\
\mathbb{T}^{\lambda_{2,1}}&=\mathbb{T}^{\lambda_{2,0}}(\omega^1)\\
\mathbb{T}^{\lambda_{2,2}}&=\mathbb{T}^{\lambda_{2,0}}(\omega^2)\\
\mathbb{T}^{\lambda_{3,0}}&=\mathbb{T}^{\lambda_{3,0}}(\omega^0)\\
\mathbb{T}^{\lambda_{3,1}}&=\mathbb{T}^{\lambda_{3,0}}(\omega^1)\\
\mathbb{T}^{\lambda_{3,2}}&= \mathbb{T}^{\lambda_{3,0}} (\omega^2) 
\end{align*}
and the columns of the following matrix form a basis of eigenvectors of $P(u,\pi_3),$

\begin{equation}\label{eigenvectorstotalxx}
 \left(
\begin{array}{ccccccccc}
1 &\omega & \omega^2 &0 &0      &0        &0  &0      &0\\
0 &0      &0         &1 &\omega & \omega^2&0  &0      &0\\
0 &0      &0         &0 &0      & 0       &1  &\omega & \omega^2\\
\frac{\lambda_{1,0}}{\mbox{\rm{prod}}_{1}(1)} &\frac{\lambda_{1,0}\omega^2}{\mbox{\rm{prod}}_{1}(1)}& \frac{\lambda_{1,0}\omega}{\mbox{\rm{prod}}_{1}(1)} &0 &0      &0        &0  &0      &0\\
0 &0      &0         &\frac{\lambda_{2,0}}{\mbox{\rm{prod}}_{1}(2)} &\frac{\lambda_{2,0}\omega^2}{\mbox{\rm{prod}}_{1}(2)}& \frac{\lambda_{2,0}\omega}{\mbox{\rm{prod}}_{1}(2)} &0  &0      &0\\
0 &0      &0         &0 &0      & 0       &\frac{\lambda_{3,0}}{\mbox{\rm{prod}}_{1}(3)} &\frac{\lambda_{3,0}\omega^2}{\mbox{\rm{prod}}_{1}(3)}& \frac{\lambda_{3,0}\omega}{\mbox{\rm{prod}}_1(3)} \\
\frac{\lambda_{1,0}^2}{\mbox{\rm{prod}}_{2}(1)} &\frac{\lambda_{1,0}^2\omega^3}{\mbox{\rm{prod}}_{2}(1)}& \frac{\lambda_{1,0}^2}{\mbox{\rm{prod}}_{2}(1)} &0 &0      &0        &0  &0      &0\\
0 &0      &0         &\frac{\lambda_{2,0}^2}{\mbox{\rm{prod}}_{2}(2)} &\frac{\lambda_{2,0}^2\omega^3}{\mbox{\rm{prod}}_{2}(2)}& \frac{\lambda_{2,0}^2}{\mbox{\rm{prod}}_{2}(2)} &0  &0      &0\\
0 &0      &0         &0 &0      & 0       &\frac{\lambda_{3,0}^2}{\mbox{\rm{prod}}_{2}(3)} &\frac{\lambda_{3,0}^2\omega^3}{\mbox{\rm{prod}}_{2}(3)}& \frac{\lambda_{3,0}^2}{\mbox{\rm{prod}}_{2}(3)} \\
\end{array}
\right)
\end{equation}

}
\endproof
\end{example}

\noindent \begin{proposition} \label{Prop
10}
The set 
\begin{equation*}\label{TUs-2}
\{\mathbb{T}^{\lambda_{t,\ell}}\,:\, t=1,2,\ldots, s\textit{ and }\ell=0,1,\ldots,k_0-1\}
\end{equation*}
forms a basis of eigenvectors of $C(u,\pi_s)$.
\end{proposition}
\begin{proof}
Consider the matrix $T$ with columns $$\mathbb{T}^{\lambda_{1,0}},\ldots, \mathbb{T}^{\lambda_{1,k_0-1}}, \mathbb{T}^{\lambda_{2,0}}\ldots,\mathbb{T}^{\lambda_{2,k_0-1}},\ldots,\mathbb{T}^{\lambda_{s,0}},\ldots, \mathbb{T}^{\lambda_{s,k_0-1}},
$$
respectively. Then we have:
\begin{align*}
 T^{-1}C(u,\pi_s)T=& \sum_{r=0}^{k}c_r T^{-1}P(u,\pi_s)^rT\\
 &=\sum_{r=0}^{k}c_r (T^{-1}P(u,\pi_s) T)^r\\
 =& \bigoplus\limits_{1\leq t \leq s} \sum_{r=0}^{k}c_r
\mbox{\,\rm{diag}\,}((\lambda_{t,0})^{r},(\lambda_{t,1})^r,\ldots,(\lambda_{t,k_0-1})^r )\\
=&\bigoplus\limits_{1\leq t \leq s }
\mbox{\,\rm{diag}\,}\left(\sum_{r=0}^{k}c_r\lambda_{t,0}^r,
\sum_{r=0}^{k}c_r(\lambda_{t,0}\omega)^r,\ldots,\sum_{r=0}^{k}c_r(\lambda_{t,0}\omega^{(k_0-1)})^r\right).
\end{align*}
\end{proof}

\begin{example}
\rm{
    Let $m=9, s=3,$ and $ C(u,\pi_3)=\sum_{r=0}^{3}c_r P(u,\pi_s)^r$ where     
    $c_r=1-(r-1)\mbox{\rm{i}}$, $r=0,1,2,3.$ and $u= (i,-1,-i,1,i,-1,-i,1,i)$.  Then
    $$
    C(u,\pi_3)=\left(\begin{array}{ccccccccc}
2+\mbox{\rm{i}} &  0  &  0  &  \mbox{\rm{i}}  &  0  &  0  & -1+\mbox{\rm{i}} & 0   &  0  \\
 0   & 3-2\mbox{\rm{i}} &  0  &  0  & -1  &  0  &   0  & 1-\mbox{\rm{i}} &  0\\ 
 0 &  0  &  -3\mbox{\rm{i}}  & 0&0 & -\mbox{\rm{i}} &0 &0 & -1+\mbox{\rm{i}} \\
 1-\mbox{\rm{i}} &  0  &  0  & 2+\mbox{\rm{i}}&0 & 0 &1 &0 & 0 \\
0& -1+\mbox{\rm{i}}&0&0& 3-2\mbox{\rm{i}}&0&0 & i &0\\
0& 0& 1-\mbox{\rm{i}}&0&0&-3\mbox{\rm{i}} & 0&0&- 1\\
-\mbox{\rm{i}} &0&0& 1+\mbox{\rm{i}}& 0&0 & 2+\mbox{\rm{i}}&0&0 \\
0& 1& 0&0 & -1- \mbox{\rm{i}}&0&0 &3-2\mbox{\rm{i}}&0\\
0&0& \mbox{\rm{i}}&0&0 & 1+\mbox{\rm{i}}&0&0&-3\mbox{\rm{i}}
    \end{array}\right)
    $$

Let $T$ be the matrix as in (\ref{eigenvectorstotalxx}), that is:
$$
T=\left(\begin{array}{ccccccccc}
   1 & \omega & \omega^2 &0&0&0&0&0&0 \\
   0 &0&0&  1& \omega & \omega^2 &0&0&0\\
   0&0&0&0&0&0&1 & \omega & \omega^2\\
   -\mbox{\rm{i}}& -\mbox{\rm{i}}\omega^2 & -\mbox{\rm{i}}\omega &0&0&0&0&0&0 \\
    0 &0&0&-\mbox{\rm{i}} & -\mbox{\rm{i}}\omega^2 & -\mbox{\rm{i}}\omega &0&0&0\\
    0&0&0&0&0&0& -\mbox{\rm{i}}& -\mbox{\rm{i}}\omega^2 & -\mbox{\rm{i}}\omega\\
    -\mbox{\rm{i}}&-\mbox{\rm{i}}&-\mbox{\rm{i}}&0&0&0&0&0&0\\
    0&0&0&-\mbox{\rm{i}}&-\mbox{\rm{i}}&-\mbox{\rm{i}}&0&0&0\\
    0&0&0&0&0&0&-\mbox{\rm{i}}&-\mbox{\rm{i}}&-\mbox{\rm{i}}\\
\end{array}\right).
$$
Then,
$$
T^{-1}C(u, \pi_3)T= W_1 \oplus W_2 \oplus W_3, $$
with 
\begin{eqnarray*}
    W_1& =&
\mbox{\,\rm{diag}\,}\left(\sum_{r=0}^{3}
c_r(\lambda_{1,0})^r,\sum_{r=0}^{3}
c_r(\lambda_{1,0}\omega)^r,\sum_{r=0}^{3}c_r(\lambda_{1,0}\omega^{2})^r\right) ,\\
W_2 &=&
\mbox{\,\rm{diag}\,}\left(\sum_{r=0}^{3}c_r(\lambda_{2,0})^r,
\sum_{r=0}^{3}c_r(\lambda_{2,0}\omega)^r,\sum_{r=0}^{3}c_r(\lambda_{2,0}\omega^{2})^r\right),\\
W_3&=&
\mbox{\,\rm{diag}\,}\left(\sum_{r=0}^{3}c_r
(\lambda_{3,0})^r,
\sum_{r=0}^{3}c_r(\lambda_{3,0}\omega)^r,
\sum_{r=0}^{3}c_r(\lambda_{3,0}\omega^{2})^r\right),
\end{eqnarray*}
where $c_0= 1-i, c_1= 1, c_2= 1- \mbox{\rm i}, c_3= 1-2 \mbox{\rm i}$, and  
\begin{eqnarray*}
\lambda_{1,0}^3=\mbox{\rm prod}_3(1)&=&u_1 u_4 u_7 =1 \\
\lambda_{2,0}^3 =\mbox{\rm prod}_3(2)&=& u_2 u_5 u_8= -i \\
\lambda_{3,0}^3= \mbox{\rm prod}_3(3)&=& u_3 u_6 u_9=-1.
\end{eqnarray*} 
}
\end{example}

\noindent \textbf{Acknowledgments}.
We thank the referee for a number of very useful comments and suggestions that improved the final version of the paper.\\
\\
\noindent Enide Andrade was support by the Portuguese Foundation for Science and Technology (FCT-Funda\c c\~ao para a Ci\^encia e a Tecnologia) through CIDMA and projects UIDB/04106/2020 and UIDP/04106/2020. Cristina Manzaneda was partially supported by VRIDT (Vicerrector\'{i}a de Investicaci\'on y Desarrollo Tectnol\'ogico), UCN, RESVRIDT 0752021 project. Dante Carrasco was partially supported by FONDECYT project 1181061, Agencia Nacional de Investigaci\'on y Desarrollo-ANID, Chile and by project 196108 GI/C, Grupo de Investigaci\'on en Sistemas Din\'amicos y Aplicaciones-GISDA, UBB, Chile.

\end{document}